\documentclass[a4paper]{amsart}
\usepackage{amsmath,amsthm,amssymb}
\usepackage[T1]{fontenc}
\usepackage{url}
\usepackage{amsmath,amsthm,amssymb}
\usepackage{verbatim}
\usepackage{mathrsfs}
\usepackage{graphics}
\usepackage{graphicx}
\usepackage{subfigure}
\usepackage{amsmath}

\newtheorem{theorem}{Theorem}[section]
\newtheorem*{theorem*}{Theorem}
\newtheorem{proposition}[theorem]{Proposition}
\newtheorem{lemma}[theorem]{Lemma}

\theoremstyle{remark}
\newtheorem{remark}[theorem]{Remark}
\newtheorem{definition}{Definition}

\newtheorem*{notation}{Notation}

\newcommand{\vep}{\varepsilon}
\newcommand{\R}{{\mathbb{R}}}

\newcommand{\C}{{\mathbb{C}}}
\newcommand{\Z}{{\mathbb{Z}}}
\newcommand{\N}{{\mathbb{N}}}
\newcommand{\T}{{\mathbb{T}}}

\newcommand{\Trans}{\operatorname{Aut}}

\newcommand{\Sing}{\Sigma}

\newcommand{\hol}{\operatorname{hol}}

\newcommand{\ind}{\operatorname{\textbf{1}}}
\newcommand{\modulo}{\operatorname{ mod}}

\begin{document}
\title[Ergodic directions for billiards in a strip]{Ergodic directions for billiards
in a strip with periodically located obstacles}
\author[K. Fr\k{a}czek \and C. Ulcigrai]{Krzysztof Fr\k{a}czek \and Corinna Ulcigrai}

\address{Faculty of Mathematics and Computer Science, Nicolaus
Copernicus University, ul. Chopina 12/18, 87-100 Toru\'n, Poland}
 \email{fraczek@mat.umk.pl}
\address{Department of Mathematics\\
University Walk, Clifton\\
Bristol BS8 1TW, United Kingdom}
\email{corinna.ulcigrai@bristol.ac.uk}
\date{\today}

  \keywords{}
\maketitle
\begin{abstract}
We study the size of the set of ergodic directions for the
directional billiard flows on the infinite band $\R\times [0,h]$
with periodically placed linear barriers of length $0<\lambda<h$.
We prove that the set of ergodic directions is always uncountable.
Moreover, if $\lambda/h\in(0,1)$ is rational the
Hausdorff dimension of the set of ergodic directions is greater
than $1/2$. In both cases (rational and
irrational) we construct explicitly some sets of ergodic
directions.
\end{abstract}

\section{Introduction}
In this paper we consider the following infinite periodic billiard, whose ergodic properties have been object of recent investigation (see e.~g.~\cite{BKM,  Fr-Ul, Hu-We}). Let $T(h,a,\lambda)$ be the billiard table (shown in  Figure \ref{fig_bil}) given by an infinite band $\R\times
[0,h]$ with periodically placed linear barriers of length
$0<\lambda<h$ handling from the lower side of the band
perpendicularly, that is:
\[
T(h,a,\lambda)=(\R\times [0,h])\setminus (a\Z\times[0,\lambda]),
\]
A \emph{billiard trajectory} is the trajectory of a point-mass which moves freely inside the table on
segments of straight lines and undergoes elastic collisions (angle
of incidence equals to the angle of reflection) when it hits the
boundary of the table.
\begin{figure}[h]
\includegraphics[width=0.6\textwidth]{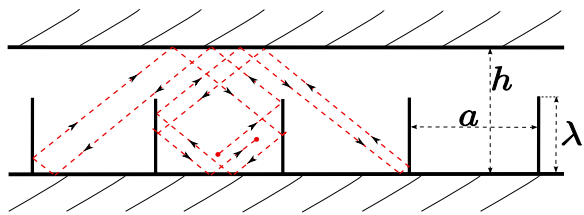}
\caption{Billiard flow on $T(h,a,\lambda)$.}\label{fig_bil}
\end{figure}
The billiard flow $(\varphi_t)_{t \in \R}$ is defined on the subset of
 the phase space $T(h,a,\lambda) \times S^1$  that consists of the points $(x, \theta) \in T(h,a,\lambda) \times S^1$ such
that if $x$ belongs to the boundary of $T(h,a,\lambda)$ then
$\theta$ is an inward direction. For $t \in \R$ and  $(x,\theta )$
in the domain of $(\varphi_t)_{t\in\R}$,  $\varphi_t$ maps $ (x,
\theta)$ to $\varphi_t(x,\theta) = (x_t,\theta_t)$, where $x_t$ is
the point reached after time $t$ by flowing at unit speed along
the billiard trajectory starting at $x$ in direction $\theta$ and
$\theta_t$ is the tangent direction to the trajectory at $x_t$. A
similar billiard in a  semi-infinite band  was studied in
\cite{BKM} in the context of perfect \emph{retroreflectors}. For a
survey on billiards in finite and infinite polygons we
refer the reader to \cite{Gu,Gut}.

Denote by $\Gamma$ the $4$-element group of isometries of $S^1$
generated by the reflections $\theta\mapsto\overline{\theta}$,
$\theta\mapsto-\overline{\theta}$. For every direction $\theta\in
S^1$ the billiard flow on $T(h,a,\lambda)$ has the invariant
subset $T(h,a,\lambda)\times (\Gamma\theta)$ in the phase space.
The billiard flow $(\varphi^{\theta}_t)_{t\in\R}$ restricted to
this set preserves the product of the Lebesgue measure on
$T(h,a,\lambda)$ and the counting measure on the orbit
$\Gamma\theta$.

In  \cite{Fr-Ul} we proved the following result:
\begin{theorem*} [\cite{Fr-Ul}]
If $\lambda/h$ is rational or
belongs to a set $\Delta\subset(0,1)$ of full Lebesgue measure
then for almost every direction $\theta$ the billiard flow
$(\varphi^{\theta}_t)_{t\in\R}$ on $T(h,a,\lambda)$ is \emph{not ergodic}.
\end{theorem*}
It is hence natural to ask whether there are exceptional ergodic directions. Hubert and Weiss proved in \cite{Hu-We} that if $\lambda/h$ is rational then
the set of ergodic directions contains a dense $G_\delta$ set.  The aim of this paper is to prove the existence of ergodic directions for all irrational
values of the relative  slit length $\lambda/h$. In addition, we also study the size of this exceptional set of ergodic directions in the rational case.
More precisely, in Section~\ref{section:irrational}, we prove that the set of ergodic directions is uncountable when $\lambda/h$ is irrational (see
Theorem~\ref{theorem:irrational}) and for rational $\lambda/h$ we prove that its Hausdorff dimension is greater than $1/2$, (see
Theorem~\ref{theorem:rational}).

In both cases (rational and irrational) we give an explicit
construction of ergodic directions by specifying their continued
fraction expansions. The proofs use an ergodicity criterion from
\cite{Hu-We} (Theorem~\ref{thm:Hu-We}) based on an approximation
of $\theta$ by directions with infinite strips. The main idea is
to study the action of $SL(2,\Z)$ on homology and to cleverly
exploit the symmetries of the system to construct infinite strips.

Combining our results from Section~\ref{section:rational} with the
approach introduced recently by Hooper in \cite{Ho1} one might be
able to describe all invariant ergodic Radon measures for
$(\varphi^{\theta}_t)_{t\in\R}$ whenever $\lambda/h$ is rational
and $\theta$ belongs to a set of positive Hausdorff dimension.

\section{Background material}

\subsection{Directional flows on translation surfaces and
$\Z$-covers}\label{defnmain:sec}
The study of directional billiard flows on any
rational polygon (not necessary compact) can be reduced via an \emph{unfolding} procedure
(introduced by Katok and Zemlyakov in \cite{Ka-Ze}) to the study of
directional flows on a translation surface.  The translation
surface corresponding to the table $T(h,a,\lambda)$ will be
described in Section~\ref{billtotrans}. In this section we briefly
recall some basic definitions related to the notion of translation
surface.

Let $M$ be an oriented surface  (not necessarily compact). A
translation surface $(M,\omega)$ is a complex structure on $M$
together with an nonzero \emph{Abelian differential} $\omega$,
that is a non-zero holomorphic $1$-form.  Let
$\Sing=\Sing_\omega\subset M$ be the set of zeros of $\omega$.
For every $\theta\in S^1 = \R/2\pi \Z $ denote by
$X_\theta=X^{\omega}_\theta$ the directional vector field in
direction $\theta$ on $M\setminus\Sing$, defined by
$i_{X_\theta}\omega =e^{i\theta}$. Then the corresponding
directional flow
$(\phi^{\theta}_t)_{t\in\R}=(\phi^{\omega,\theta}_t)_{t\in\R}$
(also known as \emph{translation flow}) on $M\setminus\Sing$
preserves the area form
$\nu_{\omega}=\frac{i}{2}\omega\wedge\overline{\omega}=\Re(\omega)\wedge\Im(\omega)$.
We will denote by $A(\omega):=\nu_\omega (M)$ the area of the
surface.

Let $(M,\omega)$ be a compact connected translation surface.
Denote by $\langle \,\cdot\,,\,\cdot\,\rangle :H_1(M,\Z)\times
H_1(M,\Z)\to\Z$ the algebraic intersection form.

Recall that a {\em $\Z$-cover} of $M$ is a surface $\widetilde{M}$
with a free totally discontinuous action of the group $\Z$ such
that the quotient manifold $\widetilde{M}/\Z$ is homeomorphic to
$M$. The map $p:\widetilde{M}\to M$ obtained by composition of the
projection $\widetilde{M}\to\widetilde{M}/\Z$ and the
homeomorphism $\widetilde{M}/\Z\to M$ is called a {\em covering
map}. Denote by $\widetilde{\omega}$ the pullback of the form
$\omega$ by the map $p$. Then $(\widetilde{M},\widetilde{\omega})$
is a translation surface as well. The translation flow on
$(\widetilde{M},\widetilde{\omega})$ in direction $\theta$ will be
denoted by $(\widetilde{\phi}^\theta_t)_{t\in\R}$.

All $\Z$-covers of $M$ (up to isomorphism) are in one-to-one
correspondence with homology classes in $H_1(M,\Z)$. The
$\Z$-cover $\widetilde{M}_\gamma$ determined by $\gamma\in
H_1(M,\Z)$, under this correspondence, has the following
properties. If $\sigma$  is a close curve in $M$ and $[\sigma]\in
H_1(M,\Z)$, then $\sigma$  lifts to a path $\widetilde{\sigma}:
[t_0, t_1]\to \widetilde{M}_\gamma$ such that $\sigma(t_1) = n
\cdot \sigma(t_0)$, where  $n:=\langle
\gamma, [\sigma] \rangle \in \mathbb{Z}$ and $\cdot$ denotes the action of $\Z$ on
$(\widetilde{M}_\gamma,\widetilde{\omega}_\gamma)$ by deck
transformations.

Denote by $\hol:H_1(M,\Z)\to \C$ the {\em holonomy map}, i.e.\
$\hol(\gamma)=\int_\gamma\omega$ for every $\gamma\in H_1(M,\Z)$.
As  shown by Hooper
and Weiss (see Proposition~15   in \cite{Ho-We}), $\hol(\gamma)=0$
if and only if for every $\theta\in S^1$ such that
$({\phi}^\theta_t)_{t\in\R}$ is ergodic, the flow
$(\widetilde{\phi}^\theta_t)_{t\in\R}$ on the $\Z$-cover
$(\widetilde{M}_\gamma,\widetilde{\omega}_\gamma)$  is recurrent.
For this reason, following \cite{Ho-We}, if $\gamma \in H_1(M,\Z)$ has $\hol(\gamma)=0$ we say that the $\Z$-cover
$(\widetilde{M}_\gamma,\widetilde{\omega}_\gamma)$ of the
translation surface $(M,\omega)$ given by $\gamma$
is {\em recurrent}.

\subsection{From billiard flows to translation flows on translation
surfaces}\label{billtotrans}
Fix parameters $(h,a,\lambda)$ and a
direction $\theta\in S^1$. One can verify, using the
unfolding process first described in \cite{Ka-Ze}, that  the flow
$(\varphi^{\theta}_t)_{t\in\R}$ on the table $T(h,a,\lambda)$ is
isomorphic to the directional flow
$(\widetilde{\phi}^\theta_t)_{t\in\R}$ on a non-compact
translation surface $(\widetilde{M},\widetilde{\omega})$ which is obtained gluing four
copies of $T(h,a,\lambda)$ along the segments of the same name, as shown in Figure~\ref{unfold1}.
\begin{figure}[h]
\includegraphics[width=1\textwidth]{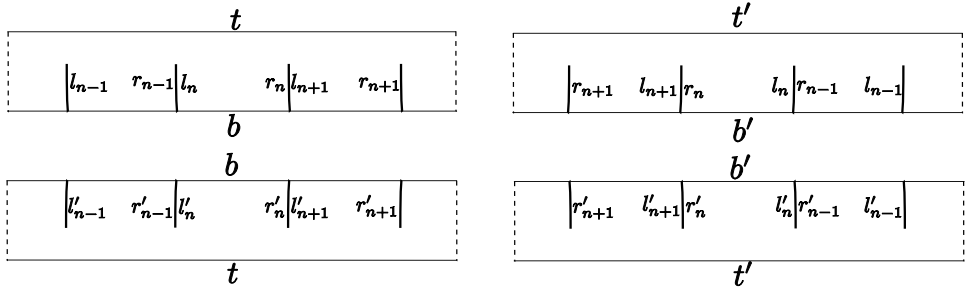}
\caption{Four copies of $T(h,a,\lambda)$}\label{unfold1}
\end{figure}
Moreover,  the surface $(\widetilde{M},\widetilde{\omega})$
can be  represented as gluing two $\Z$-periodic polygons (obtained
by gluing pairs of copies of $T(h,a,\lambda)$ along $b$
and $b'$)  as shown in the Figure~\ref{unfold2} (here $R_n=r_n\cup
r'_n$ and $L_n=l_n\cup l'_n$).
\begin{figure}[h]
\includegraphics[width=1\textwidth]{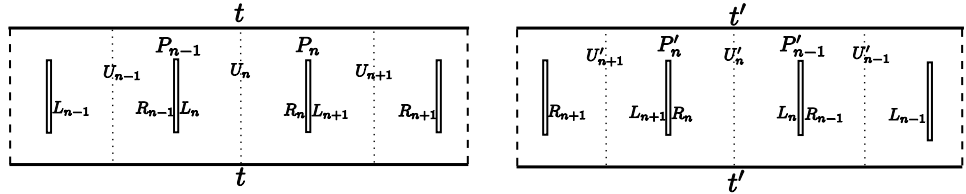}
\caption{The translation surface $(\widetilde{M},\widetilde{\omega})$}\label{unfold2}
\end{figure}
\begin{figure}[h]
\includegraphics[width=0.4\textwidth]{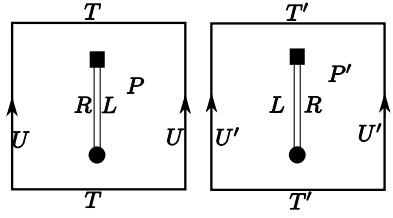}
\caption{The compact surface $(M,\omega)$}\label{unfold3}
\end{figure}
Next, let us cut these polygons into rectangles $P_n$, $P'_n$
along the  segments $U_n$, $U'_n$, $n\in\Z$ (see the
Figure~\ref{unfold3}). It follows that
$(\widetilde{M},\widetilde{\omega})$ is a $\Z$-cover of the
compact translation surface $(M,\omega)\in \mathcal{M}(1,1)$
presented in the  Figure~\ref{unfold3}. More precisely,
$(\widetilde{M},\widetilde{\omega})=(\widetilde{M}_{\gamma_0},\widetilde{\omega}_{\gamma_0})$,
where $\gamma_0=[U-U']$ and hol$(\gamma_0)=0$.
Consequently, the directional billiard flow $(\varphi^\theta_t)_{t\in\R}$ is
isomorphic to the translation flow $(\widetilde{\phi}^\theta_t)_{t\in\R}$  on the recurrent $\Z$-cover  $(\widetilde{M}_{\gamma_0},\widetilde{\omega}_{\gamma_0})$.

\subsection{Moduli space and Teichm\"uller flow}\label{Teich:sec}
Let $M$ be  a compact connected oriented surface of genus $g$ and
let $\Sing\subset M$ be a finite set with cardinality $s\in \mathbb{N}$.  Denote by
$\operatorname{Diff}^+(M,\Sing)$ the group of
orientation-preserving homeomorphisms of $M$ preserving $\Sing$.
Denote by $\operatorname{Diff}_0^+(M,\Sing)$ the subgroup of
elements $\operatorname{Diff}^+(M,\Sing)$ which are isotopic to
the identity. Let us denote by
$\Gamma(M,\Sing):=\operatorname{Diff}^+(M,\Sing)/\operatorname{Diff}_0^+(M,\Sing)$
the {\em mapping-class} group.

 Let $\kappa=(\kappa_1,\ldots,\kappa_s)$ be a family of natural numbers
such that $2g-2 = \sum_{i=1}^s\kappa_i$.
The \emph{stratum}
$\mathcal{M}(\kappa)=\mathcal{M}(M,\Sing,\kappa)$ of the {\em moduli
space of  Abelian differentials} is the space of orbits of
the natural action of   $\operatorname{Diff}^+(M,\Sing)$  on the space of all
Abelian differentials on $M$ with $s$ zeros at $\Sing$ of degrees
$\kappa_1, \dots, \kappa_s$.  The \emph{stratum} $\mathcal{Q}(M,\Sing,\kappa)$ of the {\em Teichm\"uller space of
Abelian differentials} is the space of orbits of the natural
action of $\operatorname{Diff}_0^+(M,\Sing)$ on the space of all
Abelian differentials on $M$ with $s$ zeros at $\Sing$ of degrees
$\kappa_1, \dots, \kappa_s$. Thus
$\mathcal{M}(M,\Sing,\kappa)=\mathcal{Q}(M,\Sing,\kappa)/\Gamma(M,\Sing)$.

The group $GL(2,\R)$ acts naturally on
$\mathcal{Q}(M,\Sing,\kappa)$ and $\mathcal{M}(M,\Sing,\kappa)$, by postcomposition with the charts defined by local primitives of
the holomorphic $1$-form.  The Abelian differential
obtained acting by $g \in GL(2,\R)$ on
$\omega$ will be denoted by   $g\cdot \omega$. The {\em Teichm\"uller flow} $(G_t)_{t\in\R}$ is the
restriction of this action to the diagonal subgroup
$(\operatorname{diag}(e^t,e^{-t}))_{t\in\R}$ of $GL(2,\R)$ on
$\mathcal{Q}(M,\Sing,\kappa)$ and $\mathcal{M}(M,\Sing,\kappa)$.

\subsection{Ergodicity for $\Z$-periodic surfaces}
In this section we formulate a
result from  \cite{Hu-We} which provides an effective method to prove the ergodicity of translation flows
on recurrent $\Z$-covers of compact translation surfaces and will be exploited   to prove our main results.

Let $(M,\omega)$
be a compact connected translation surface. Let
$(\widetilde{M}_\gamma,\widetilde{\omega}_\gamma)$ be one of its
recurrent $\Z$-covers. Suppose that $C\subset M$ is a cylinder.
Let $\delta(C)\in H_1(M,\Z)$ be the homology class of any core curve of the cylinder $C$.
We will use the following notation (introduced in \cite{Hu-We}):
\[k(C):=\langle \delta(C),\gamma\rangle\in\Z,\quad
v(C):=\operatorname{hol}(\delta(C))\in\R^2,\quad
A(C)>0\ \text{ is the area of }\ C.\]
Note that if $k(C)\neq 0$ then the lift $\widetilde{C}_\gamma\subset \widetilde{M}_\gamma$ of $C$ to
the $\Z$-cover $\widetilde{M}_\gamma$ is an infinite strip.
\begin{definition}[see \cite{Hu-We}]
A direction $\theta\in S^1$ is \emph{well approximated} by strips
of the surface $(\widetilde{M}_\gamma,\widetilde{\omega}_\gamma)$
if there exist  $\vep$, $c>0$, $k_\theta\in\Z\setminus\{0\}$ and
infinitely many strips $\widetilde{C}\subset \widetilde{M}_\gamma$
for which
\begin{equation}\label{eq:defwast}
k(C)=k_\theta,\quad A(C)>c\quad\text{ and }\quad|\left( \cos \theta, \sin \theta\right )\wedge
v(C)|\leq (1-\vep)\frac{A(C)}{2\|v(C)\|}.
\end{equation}
\end{definition}

The following result follows directly from the proof of Theorem 1
in \cite{Hu-We} (more precisely from Claim 12).

\begin{theorem}\label{thm:Hu-We}
Suppose that $\theta\in S^1$ is an ergodic direction for the
translation flow on $(M,\omega)$. If $\theta\in S^1$ is  well
approximated  by strips of the surface
$(\widetilde{M}_\gamma,\widetilde{\omega}_\gamma)$ with
$k_\theta=\pm1$ then the flow
$(\widetilde{\phi}^\theta_t)_{t\in\R}$ on
$(\widetilde{M}_\gamma,\widetilde{\omega}_\gamma)$ is ergodic.
\end{theorem}

In order to prove ergodicity of the translation flow, which is required to apply the above Theorem, the following result from \cite{Ma}, know as \emph{Masur's criterion}, will be helpful. 

\begin{theorem}[Masur's criterion \cite{Ma}]\label{thm:masur}
Let $(M,\omega)\in \mathcal{M}(\kappa)$ be a compact translation
surface. Let $g\in SL(2,\R)$ be  an element that maps the direction $\theta$ to the vertical direction.
Suppose that there exists a bounded subset $B\subset
\mathcal{M}(\kappa)$ and a sequence $t_n\to+\infty$ such that $
G_{t_n}(g \cdot (M,\omega))\in B$ for all $n\in\N$. Then the
directional flow $(\phi^\theta_t)_{t\in\R}$  on $(M,\omega)$ is
uniquely ergodic.
\end{theorem}

\section{Construction of ergodic directions}\label{sec:sl2zaction}
In this section we describe a procedure to construct strips on  the translation surface
$(\widetilde{M}_{\gamma_0},\widetilde{\omega}_{\gamma_0})$
(see the end of Section~\ref{billtotrans}) which will allow us to apply
Theorem~\ref{thm:Hu-We}. We will study the
behavior of the  $SL(2,\Z)$ orbit of $\gamma_0\in H_1(M,\Z)$ for
the $SL(2,\Z)$-action induced on $H_1(M,\Z)$. For simplicity, we
always assume that $a=1=2h$, then $\lambda=|L|/2=|R|/2\in(0,1/2)$.

Let $\T^2_0$ denote the set
\[\T^2_0:=[-1/2,1/2)\times[-1/2,1/2)\setminus\{(0,0),(-1/2,-1/2),(-1/2,0),(0,-1/2)\}.\]
For $z \in \mathbb{T}^2_0$, let $M(z) \in \mathcal{M}(1,1) $ be the translation surface drawn
in Figure~\ref{defemy}. We will distinguish between the singular points ({\Large$\bullet$} and
{\scriptsize$\blacksquare$} in Figure~\ref{defemy}) and    $z=(x,y)$ will denote the position of the singular point
{\scriptsize$\blacksquare$} while $-z$ will be  the position of the
singular point {\Large$\bullet$}.
\begin{figure}[h]
\includegraphics[width=0.4\textwidth]{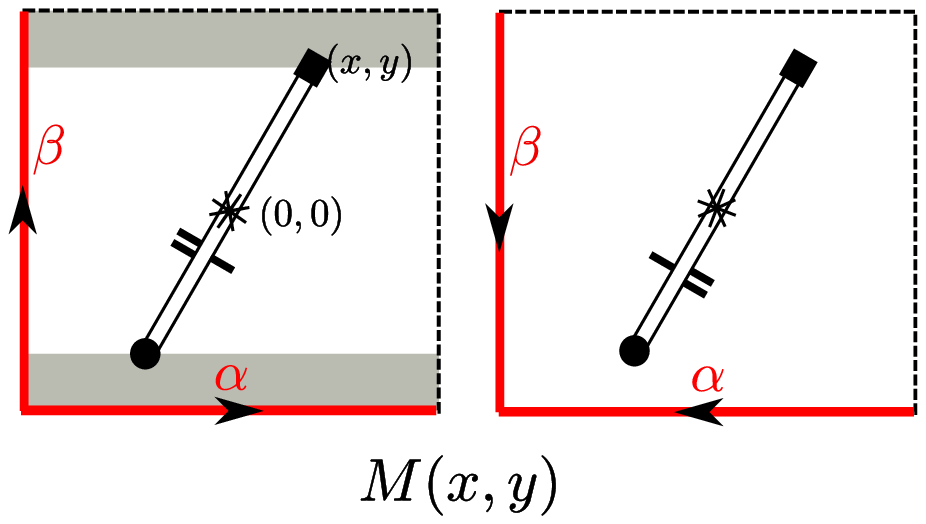}
\caption{}\label{defemy}
\end{figure}
Let $\mathscr{L}\subset\mathcal{M}(1,1)$ be the  locus   containing all the surfaces $M(z), z \in \mathbb{T}^2_0$. Notice that the group $\Trans(M(z))$ of automorphisms of the
translation surface fixing singular points consists of the identity map $id$ and of the map
$\tau:M(z)\to M(z)$ which exchanges the two squares (tori) in $M(z)$ by  translations.

For every $M(z)\in
\mathscr{L}$ let us fix the basis (called a standard basis)
$\{\alpha,\beta\}$ of the subgroup
\[H^{(0)}_1(M(z),\Z)=
\{\gamma\in H_1(M(z),\Z):\tau_*\gamma=-\gamma\}\] as in
Figure~\ref{defemy}. Then $(M,\omega)=M(0,\lambda)\in
\mathscr{L}$ and $\gamma_0=\beta$.
Notice that the locus $\mathscr{L}$ is $GL(2,\Z)$-invariant.

The group $GL(2,\R)$ acts naturally on $\mathbb{R}^2$ by matrix
multiplication. We will denote the image of $z=(x,y) \in
\mathbb{R}^2$ under $g \in GL(2,\R)$  by $g (x,y)$ or $g z $. This
action induces also a natural action  $GL(2,\Z) \subset GL(2,\R)$
on the torus $\T^2_0$. We will also denote by $g z\in \T^2_0$ the
image of $z\in \T^2_0$  by  the automorphism of $\mathbb{T}^2 _0 $
induced by $g$.
One can show that for every $g\in GL(2,\Z)$:
\begin{equation}\label{eq:fund}
g\cdot M(z)\text{ is identified with }M(gz)\text{ in
}\mathcal{M}(1,1).
\end{equation}
In order to verify (\ref{eq:fund}), it suffices to
check it on generators of $GL(2,\Z)$ and this is obtained as a byproduct of the proof of Lemma~\ref{lem:basic}. Since
$\Trans(M(z))=\{id,\tau\}$, there exist exactly two affine maps
$\zeta^g,\tau\circ\zeta^g: M(z)\to g\cdot M(z)$ whose derivatives are
equal to $g$.
Thus, if we denote by $\zeta_*:H^{(0)}_1(M(z),\Z)\to H^{(0)}_1(
M(z'),\Z)$ the action on homology induced by a diffeomorphism $\zeta: M(z) \to M(z')$, we have $(\tau\circ\zeta^g)_*=-\zeta^g_*$. Therefore, the
induced homology action $g_*(z):H^{(0)}_1(M(z),\Z)\to
H^{(0)}_1(g\cdot M(z),\Z)$ is well defined up to $\pm$, i.e.
$g_*(z)=\pm\zeta^g_*$. Representing this action in standard basis we
obtain a matrix, denoted by $g_*(z)$, which is an element of
$PGL(2,\Z)$. Finally note that for all $g_1,g_2\in GL(2,\Z)$ we
have
\begin{equation}\label{eq:comphom}
(g_1\cdot g_2)_*(z)=(g_1)_*(g_2z)\cdot(g_2)_*(z).
\end{equation}

\subsection{Construction of strips and an ergodicity criterion}\label{rem:procstrips}

We now present the procedure of construction of infinite family of
infinite strips required by Theorem~\ref{thm:Hu-We}.

For any $z=(x,y)\in\T^2_0$ let us consider the cylinder
\[C_z=[-1/2,1/2)\times ([y,1/2)\cup[-1/2,-y])\quad\text{ in }\quad M(z),\]
which is shaded in Figure~\ref{defemy}. The holonomy vector of
the core $\delta(C_z)$ of $C_z$ is $v(C_z)=(1,0)$ and
$A(C_z)=1-2y$, moreover $\langle\delta(C_z),\beta\rangle=1$.

For any $g\in SL(2,\Z)$ let $z_g=(x_g,y_g):=g^{-1}(z)\in\T^2_0$ and
$\zeta^g:M(z_g)\to M(z)$ be an affine transformation whose
derivative is $g\in SL(2,\Z)$. Then $C_g:=\zeta^g(C_{z_g})\subset
M(z)$ is a cylinder such that
$\delta(C_g)=\zeta^g_*(\delta(C_{z_g}))\in H_1(M,\Z)$ and
\begin{equation*}
v(C_g)=v\big(\zeta^g_*(\delta(C_{z_g}))\big)=\operatorname{hol}\big(\zeta^g_*(\delta(C_{z_g}))\big)
=D\zeta^g\operatorname{hol}\big(\delta(C_{z_g})\big)=g\,(1,0),
\end{equation*}
\begin{equation*}
k(C_g)=\langle\zeta^g_*(\delta(C_{z_g})),\beta\rangle=\langle\delta(C_{z_g}),(\zeta^g_*)^{-1}\beta\rangle,
\end{equation*}
\begin{equation*}
A(C_g)=A(\zeta^g(C_{z_g}))=A(C_{z_g})=1-2y_g.
\end{equation*}
Suppose additionally that $g_*(z_g)\beta=\beta$. Then $\zeta^g_*\beta=\pm\beta$, hence
\[k(C_g)=\langle\delta(C_{z_g}),(\zeta^g_*)^{-1}\beta\rangle=\pm\langle\delta(C_{z_g}),\beta\rangle=\pm 1.\]
In summary, the above procedure allows to
construct strips on $\widetilde{M(z)}_\beta$ satisfying \eqref{eq:defwast} whenever
\begin{equation}\label{eq:invbeta}
g_*(g^{-1}(z))\beta=\beta.
\end{equation}
Combining this with Theorem~\ref{thm:masur} (Masur's criterion) and Theorem~\ref{thm:Hu-We}
(Hubert-Weiss criterion) yields the following criterion for the ergodicity
of directional flows on the surface $\widetilde{M(z)}_\beta$.

\begin{theorem}\label{theorem:ergodicity}
Suppose that $z_0\in\T^2_0$, $\alpha=[0;a_1,a_2,a_3,\ldots]$,
$0<a<b<1/2$ and there exists an increasing sequence of even
numbers $(k_n)_{n\geq 1}$ such that
\begin{equation}\label{eq:zalerg}
\big((h^{+})^{a_{1}}\cdots(h^{-})^{a_{k_n}}\big)^{-1}\!
M(z_0)=M(z_n),\
\big((h^{+})^{a_{1}}\cdots(h^{-})^{a_{k_n}}\big)^{-1}_*(z_0)\,
\beta=\beta,
\end{equation}
and setting $(x_n,y_n)=z_n$ we have
\[a\leq y_n\leq b\quad\text{ and }\quad
a_{k_n+1}\geq \frac{2}{1-2y_n}.\]
Then the
directional flow in direction $(1,\alpha)$ on
the $\mathbb{Z}$-cover $ \widetilde{M(z_0)}_\beta$ given by $\beta$  is ergodic.
\end{theorem}

\begin{proof}
First we will show that the  directional flow in direction
$(1,\alpha)$ on the surface $M(z_0)$ is ergodic. Let
$(p_n/q_n)_{n\geq 0}$ stand for the sequence of convergents of the
continued fraction of $\alpha$. Then
\[(h^{+})^{a_{1}}\cdots(h^{-})^{a_{k_n}}=\begin{pmatrix}
q_{k_n}&q_{k_n-1}\\p_{k_n}&p_{k_n-1}
\end{pmatrix}\]
Setting
\[\sigma:=
\begin{pmatrix}\alpha&-1\\0&1/\alpha
\end{pmatrix},
\]
it is clear that $\sigma$ maps the direction of the vector $(1,\alpha)$ to the vertical direction. We will prove that the sequence
\[\big(G_{\log q_{k_n}}\cdot \sigma \cdot M(z_0)\big)_{n\geq 1}\]
is bounded in the moduli space. Indeed, by assumption,
\[G_{\log q_{k_n}}\cdot \sigma \cdot M(z_0)=\sigma_n \cdot M(z_n),\]
where
\begin{align*}
\sigma_n:&=\operatorname{diag}(q_{k_n},1/q_{k_n})\cdot
\sigma\cdot(h^{+})^{a_{1}}\cdots(h^{-})^{a_{k_n}}\\&=
\begin{pmatrix}q_{k_n}&0\\0&1/q_{k_n}
\end{pmatrix}
\begin{pmatrix}\alpha&-1\\0&1/\alpha
\end{pmatrix}
\begin{pmatrix}
q_{k_n}&q_{k_n-1}\\p_{k_n}&p_{k_n-1}
\end{pmatrix}\\&=\begin{pmatrix}
q_{k_n}(q_{k_n}\alpha-p_{k_n})&q_{k_n}(q_{k_n-1}\alpha-p_{k_n-1})\\p_{k_n}/(\alpha
q_{k_n})&p_{k_n-1}/(\alpha q_{k_n})
\end{pmatrix}.
\end{align*}
Moreover,
\[ |q_{k_n}(q_{k_n}\alpha-p_{k_n})|<|q_{k_n}(q_{k_n-1}\alpha-p_{k_n-1})|<1\]
and, since $k_n$ is by construction even, so that $p_{k_n}/q_{k_n}< \alpha$, we have
\[0<\frac{p_{k_n-1}}{\alpha q_{k_n}}<\frac{p_{k_n}}{\alpha q_{k_n}}<1.\]
Therefore, all $\sigma_n$ belong to the subset $G_0\subset
SL(2,\R)$ of matrices whose coefficients belong to $[-1,1]$. Of
course, the set $G_0$ is compact. Let us consider the set
\[B_0:=\big\{M(x,y)\in\mathscr{L}:y\in[a,b]\big\},\]
which is compact in the moduli space $\mathcal{M}(1,1)$. Thus $G_0\cdot B_0\subset\mathcal{M}(1,1)$ is  a
compact subset in the moduli space $\mathcal{M}(1,1)$ as well.
Since $\sigma_n\cdot M(z_n)\in G_0\cdot B_0$ for every natural $n$, in view of
Theorem~\ref{thm:masur}, the  directional flow in direction
$(1,\alpha)$ on the surface $M(z_0)$ is ergodic.

In the rest of the proof, Theorem~\ref{thm:Hu-We} combined with the construction described before
Theorem~\ref{theorem:ergodicity} will give the ergodicity of
directional flow in direction
$(1,\alpha)$ on the $\Z$-cover $\widetilde{M(z_0)}_\beta$.

Recall that for any $z=(x,y)\in\T^2_0$  we denote by $C_z$ the cylinder
$[-1/2,1/2)\times ([y,1/2)\cup[-1/2,-y])\subset M(z)$.
The holonomy vector of the core $\delta(C_z)$ of
$C_z$ is $v(C_z)=(1,0)$, $\operatorname{Area}(C_z)=1-2y$ and
$\langle\delta(C_z),\beta\rangle=1$.

By assumption, for every  $n\geq 1$ there exists an affine
transformation
\[\zeta_n:M(z_n)\to M(z_0)\]
whose derivative $D\zeta_n$ is $(h^{+})^{a_{1}}\cdots(h^{-})^{a_{k_n}}=\begin{pmatrix}
q_{k_n}&q_{k_n-1}\\p_{k_n}&p_{k_n-1}
\end{pmatrix}$ and
$(\zeta_n)_*\beta=\beta$. Let us consider the cylinder
$C_n:=\zeta_n(C_{z_n})\subset M(z_0)$ for which the homology
class of the core is $\delta(C_n)=(\zeta_n)_*(\delta(C_{z_n}))$. Then
\[k(C_n)=\langle\delta(C_n),\beta\rangle=\langle(\zeta_n)_*(\delta(C_{z_n})),(\zeta_n)_*(\beta)\rangle=
\langle\delta(C_{z_n}),\beta\rangle=1,\]
\[v(C_n)=(D\zeta_n) v(C_{z_n})=(D\zeta_n)(1,0)=(q_{k_n},p_{k_n}),\]
\[A(C_n)=A(C_{z_n})=1-2y_n.\]
Therefore each cylinder $C_n\subset M(z_0)$ is lifted to an
infinite strip $\widetilde{C}_n\subset \widetilde{M(z_0)}_\beta$.
Using this sequence of strips we can show that the direction
$(1,\alpha)$ is well approximated by strips. Indeed,
\[\frac{|(1,\alpha)\wedge v(C_n)|}{\|(1,\alpha)\|}=
\frac{|(1,\alpha)\wedge (q_{k_n},p_{k_n})|}{\|(1,\alpha)\|}=
\frac{|q_{k_n}\alpha-p_{k_n}|}{\|(1,\alpha)\|}<
\frac{1}{\|(1,\alpha)\|}\frac{1}{a_{k_n+1}q_{k_n}}\] and
\[\frac{A(C_n)}{2\|v(C_n)\|}\geq \frac{1-2y_n}{2q_{k_n}\|(1,\alpha)\|}.\]
Therefore,
\[\frac{|(1,\alpha)\wedge v(C_n)|}{\|(1,\alpha)\|}\leq \frac{1}{2}\frac{A(C_n)}{2\|v(C_n)\|}\]
whenever $a_{k_n+1}\geq 2/(1-2y_n)$.

The ergodicity of the directional flow in direction
$(1,\alpha)$ on the surface $\widetilde{M(z_0)}_\beta$ hence
follows directly from Theorem~\ref{thm:Hu-We}.
\end{proof}

\subsection{$SL_+(2,\Z)$-action induced on homologies}
Denote by $SL_+(2,\Z)$ the semi-group of non-negative matrices
in $SL(2,\R)$. In this section we establishes some principal rules of the
$SL_+(2,\Z)$-action induced on homologies. These rules will help us to find elements
$g\in SL_+(2,\Z)$ satisfying \eqref{eq:invbeta}, first
for irrational $\lambda$ in Section~\ref{section:irrational} and then for
rational $\lambda$ in Section~\ref{section:rational}.

Set
\[h^+:=\begin{pmatrix}1&1\\0&1\end{pmatrix},\quad h^-:=\begin{pmatrix}1&0\\1&1\end{pmatrix},\quad \omega:=\begin{pmatrix}0&1\\-1&0\end{pmatrix},\quad \vartheta:=\begin{pmatrix}0&1\\1&0\end{pmatrix}.\]
Then $SL_+(2,\Z)$ is generated by $h^+$ and $h^-$ and one can check that:
\begin{equation}\label{eq:matrixrel}
\begin{split}
& \vartheta \cdot h^{\pm} \cdot\vartheta^{-1}=h^{\mp}, \qquad \vartheta \cdot \omega^{\pm1} \cdot \vartheta^{-1}=\omega^{\mp1}, \qquad   \omega
\cdot h^{\pm}\cdot\omega^{-1}=(h^{\mp})^{-1},\\
& h^-\cdot(h^+)^{-1} \cdot h^- =\omega^{-1},
\qquad  h^+\cdot(h^-)^{-1} \cdot h^+ =\omega.
\end{split}
\end{equation}
Let us consider two involutions $-id$ and $\vartheta$ in
$GL(2,\Z)$. They generate two involutions (denoted also by $-id$
and $\vartheta$) acting on the locus $\mathscr{L}$. Geometrically,
\begin{itemize}
\item $-id$ rotates the squares of $M(z)$ by angle $\pi$, therefore
its induced action on $H^{(0)}_1(M,\Z)$ maps $\alpha$ to $-\alpha$ and
$\beta$ to $-\beta$ except when $z$ belongs to the boundary of $\mathbb{T}^2_0$;
\item  $\vartheta$ reflects the squares of $M(z)$ across their
diagonal, therefore its action on $H^{(0)}_1(M,\Z)$ exchanges the
basis elements $\alpha$ and $\beta$.
\end{itemize}
These two involutions (symmetries) will help us to describe the
action of $SL_+(2,\Z)$ on the homology level. For this purpose, we
will use the following two lemmas.
\begin{lemma}\label{cor:symetry}
Set $E:=\{(x,y)\in \T_0^2:x,y>-1/2\}$.
For every $z\in\T^2_0$ we have
\begin{equation}\label{eq:firstclaim}
\vartheta_*(z)=\vartheta\quad\text{ and } \quad(-id)_*(z)= id
\quad\text{whenever}\quad z\in E.
\end{equation}
Suppose that $g\cdot M(z)=M(z')$ for some $g\in SL(2,\Z)$. Then
\begin{align}\label{eq:vartheta}
(\vartheta \cdot g \cdot\vartheta^{-1}) \cdot M(\vartheta
z)=M(\vartheta z')\quad \text{and}\quad(\vartheta \cdot g \cdot
\vartheta^{-1})_*(\vartheta z)=\vartheta \cdot
g_*(z)\cdot\vartheta^{-1}.
\end{align}
If additionally $z,z'\in E$ then
\begin{align}\label{eq:minusidtrans}
g\cdot M(-z)=M(-z')\quad\text{ and }\quad g_*(-z)=g_*(z).
\end{align}
\end{lemma}
\begin{proof}
The first claim \eqref{eq:firstclaim} follows from the
observations formulated before the lemma. The equality $(\vartheta
\cdot g \cdot\vartheta^{-1}) \cdot M(\vartheta z)=M(\vartheta z')$
follows directly from the fact that $z'=gz$ and \eqref{eq:fund}.
In view of \eqref{eq:comphom},
\[(\vartheta \cdot g \cdot
\vartheta^{-1})_*(\vartheta z)=\vartheta_*(gz)\cdot g_*(z)\cdot
\vartheta^{-1}_*(\vartheta z).\] Since $\vartheta^{-1}=\vartheta$,
from \eqref{eq:firstclaim} we obtain \eqref{eq:vartheta}.

The equality $ g \cdot M((-id)z)=M((-id)z')$ also follows directly
from  \eqref{eq:fund} and the fact that $-id$ and $g$ commute.
Moreover, in view of \eqref{eq:comphom} and \eqref{eq:firstclaim}
applied to $z,z'\in E$,
\begin{align*}
g_*(z)&=(-id)_*(z')\cdot g_*(z)=((-id)\cdot g)_*(z)=(g\cdot
(-id))_*(z)\\&=g_*(-z)\cdot (-id)_*(z)=g_*(-z).
\end{align*}
\end{proof}

\begin{figure}[h]
\includegraphics[width=0.9\textwidth]{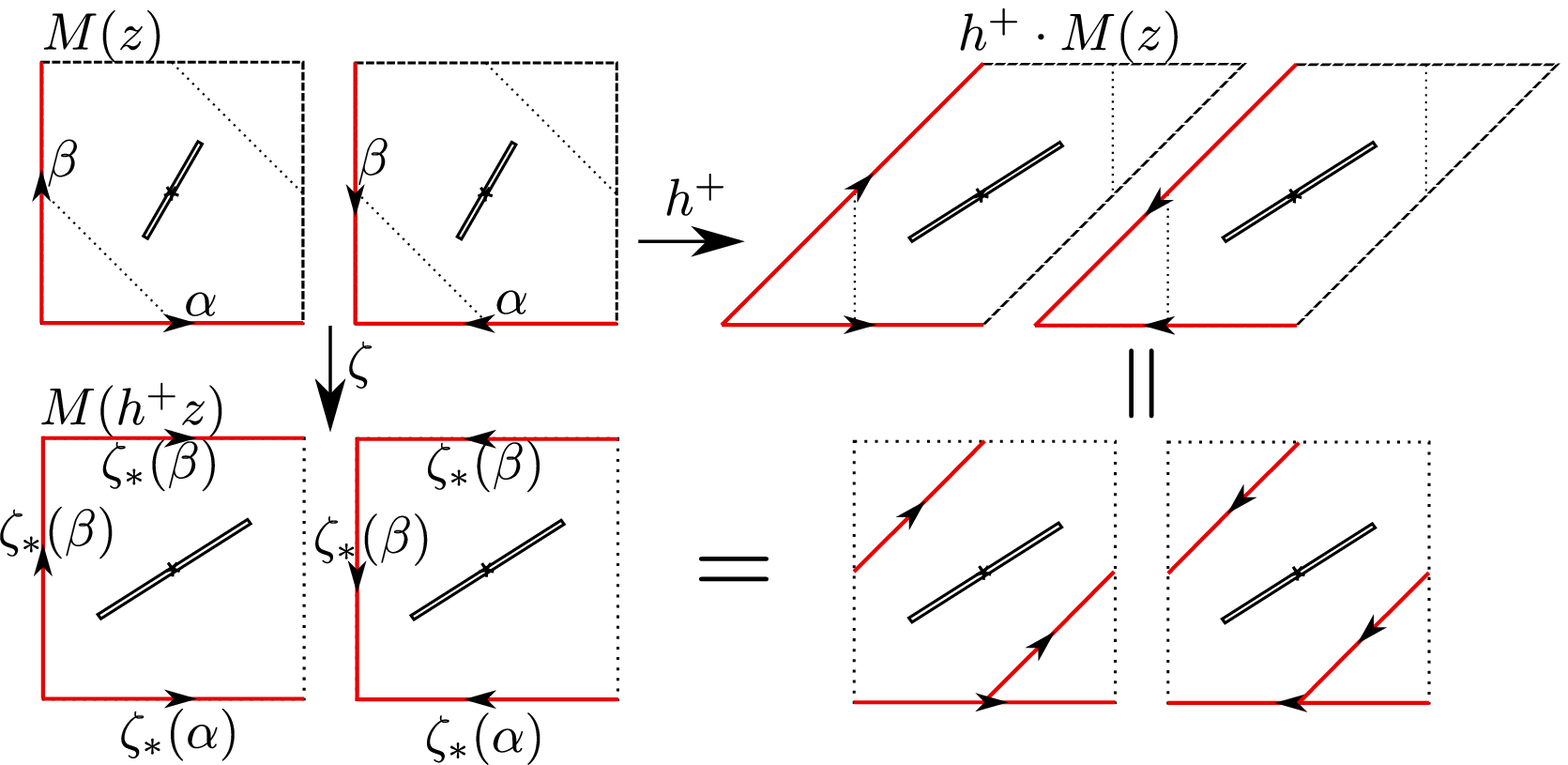}
\caption{The action of $h^+$ on $M(z)$ if $z\in S$}\label{homreg}
\end{figure}
\begin{figure}[h]
\includegraphics[width=0.9\textwidth]{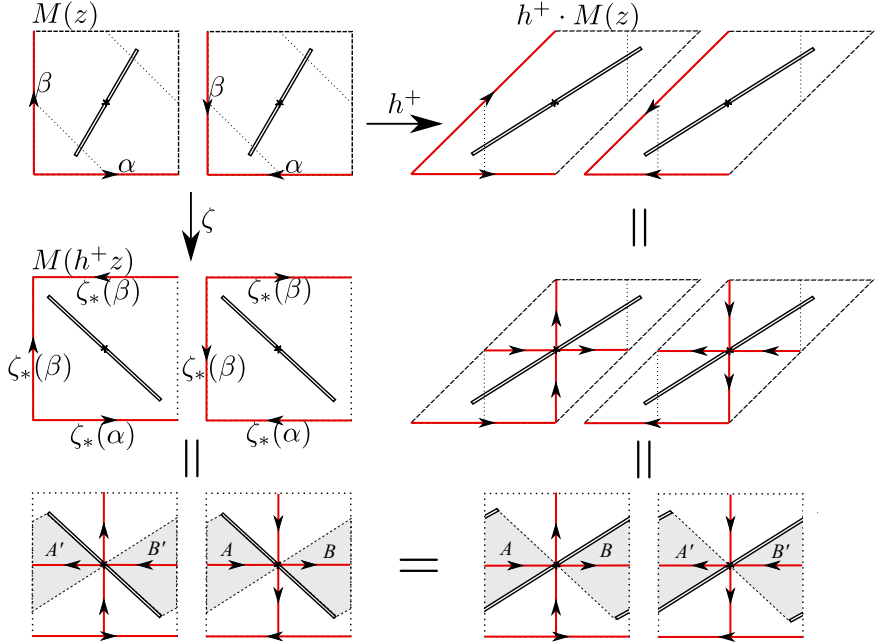}
\caption{The action of $h^+$ on $M(z)$ if $z\notin
S$}\label{homnonreg}
\end{figure}

\begin{lemma}\label{lem:basic}
Set
$S:=\{(x,y)\in \T^2_0:-1/2\leq x+y< 1/2\}$.
For every $z\in\T^2_0$ we have
\[h^{\pm}_*(z)=\left\{\begin{matrix}h^\pm & \text{ if } & z\in S,\\
(h^\pm)^{-1} & \text{ if } & z\notin S.
\end{matrix}\right.\]
\end{lemma}

\begin{proof}
In Figures \ref{homreg} and \ref{homnonreg} we present the surface $h^+\cdot M(z)$ for $z\in S$ and $z\notin S$ respectively and using cutting and pasting
we show how to represent $h^+\cdot M(z)$ as $M(h^+z)$.  The cut and paste in Figure~\ref{homreg} is straightforward. Let us explain Figure~\ref{homnonreg}:
after the linear action of $h^+$ and a first cut and paste,
we consider the shaded areas labeled by $A, A', B$ and $B'$  in Figure~\ref{homnonreg}. Recalling the gluings between slits, one can verify that $A$ and
$A'$ can be exchanged and similarly $B$ and $B'$. The surface that we obtain after this operation is one of the canonical representatives of the locus
$\mathscr{L}$, more precisely it is $M(z')$, where $z'= h^+ z$.

 Let us denote by  $\zeta=\zeta^{h^+}:M(z)\to M(h^+z)$ the affine map obtained combining the linear action of $h^+$ with the cut and paste operations. By construction, we have
$D\zeta^{h^+}=h^+$. In order to describe its induced action $\zeta^{h^+}_*$ on homology, in Figures   \ref{homreg} and \ref{homnonreg} we draw the images of the homology classes of $\alpha, \beta$ under $\zeta^{h^+}$. By changing representatives as shown in Figures \ref{homreg} and \ref{homnonreg}, one can verify that:
\[\zeta^{h^+}_*(\alpha)=\alpha,\ \zeta^{h^+}_*(\beta)=\alpha+\beta\text{ if }z\in
S,\quad \zeta^{h^+}_*(\alpha)=\alpha,\ \zeta^{h^+}_*(\beta)=\beta-\alpha\text{
if }z\notin S.\] Therefore, $(h^+)_*(z)=h^+$ if $z\in S$ and
$(h^+)_*(z)=(h^+)^{-1}$ if $z\notin S$.

In view of \eqref{eq:matrixrel} and \eqref{eq:vartheta}, for any
$z\in \T^2_0$
\[(h^-)_*(z)=(\vartheta \cdot h^+\cdot \vartheta^{-1})_*(z)=
\vartheta \cdot (h^+)_*(\vartheta^{-1} z)\cdot \vartheta^{-1}.\]
Since $S$ is symmetric with respect to the involution $\vartheta$, $z\in S$ if and only if $\vartheta^{-1} z=\vartheta z\in S$. By \eqref{eq:matrixrel} and \eqref{eq:vartheta}, it follows that $(h^-)_*(z)=h^-$ if $z\in
S$ and $(h^-)_*(z)=(h^-)^{-1}$ if $z\notin S$.
\end{proof}

In view of Lemma~\ref{lem:basic}, for every $z\in\T^2_0$ and
$n\in\Z$ there exists $m_n^\pm(z)\in\Z$ such that
\[\big((h^\pm)_*^{-n}(z)\big)^{-1}=(h^\pm)^n_*\big((h^\pm)^{-n}z\big)=(h^\pm)^{m_n^\pm(z)}.\]

\begin{lemma}\label{lem:mn}
If $z=(x,y)\in\T^2_0$ with irrational $x$ and $y$ then
\begin{itemize}
\item[(i)]  $m_n^\pm(z)\to+\infty$ as $n\to+\infty$;
\item[(ii)] the sequences $\big(m_n^\pm(z)\big)_{n\geq 1}$ take
all natural values.
\end{itemize}
\end{lemma}

\begin{proof}
We will give the proof for the matrix $h^+$. The case of the
matrix $h^-$ is similar. By Lemma~\ref{lem:basic}, for every
$z'\in\T^2_0$ we have
\[(h^+)_*(z')=(h^+)^{2\ind_{S}(z')-1},\]
where $\ind_{S}$ is the indicator function of the subset $S$. In
view of \eqref{eq:comphom}, it follows that
\begin{align*}
(h^+)^n_*\big((h^+)^{-n}z\big)&= h^+_*\big((h^+)^{-1}z\big)\cdot
h^+_*\big((h^+)^{-2}z\big)\cdots h^+_*\big((h^+)^{-n+1}z\big)\cdot
h^+_*\big((h^+)^{-n}z\big)\\
&=(h^+)^{\sum_{j=1}^n(2\ind_{S}(x-jy,y)-1)}.
\end{align*}
Therefore
\[m_n^+(z)=\sum_{j=1}^n(2\ind_{S}(x-jy,y)-1)=\sum_{j=1}^n(2\ind_{S_y}(x-jy)-1),\]
where $S_y:=\{x'\in\T:(x',y)\in S\}$ ($\T:=[-1/2,1/2)$). Since
$S_y$ is an interval of length $1-|y|$, the mean value of the
function $2\ind_{S_y}(\,\cdot\,)-1$ is $1-2|y|>0$ and, by unique
ergodicity of the rotation by $y$ on $\T$, we have
\[\frac{1}{n}\sum_{j=1}^n(2\ind_{S_y}(x-jy)-1)\to 1-2|y|>0\quad\text{as}\quad n\to+\infty.\]
It follows that
\[m_n^+(z)=\sum_{j=1}^n(2\ind_{S_y}(x-jy)-1)\to+\infty\quad\text{as}\quad n\to+\infty.\]
Since $m_0^+(z)=0$ and $m_{n+1}^+(z)=m_n^+(z)\pm 1$, the condition
(i) implies (ii).
\end{proof}

\section{Irrational $\lambda$}\label{section:irrational}
The aim of this section is to give the proof of the following
result on the existence  of ergodic directions on
$\widetilde{M(0,\lambda)}_\beta$ whenever $\lambda\in(0,1/2)$ is
irrational.
\begin{theorem}\label{theorem:irrational}
For  every irrational $\lambda\in(0,1/2)$ the set of ergodic
directions on the surface $\widetilde{M(0,\lambda)}_\beta$ is
uncountable.
\end{theorem}

This result follows from Theorem~\ref{theorem:ergodicity}
and the following auxiliary lemma.

\begin{lemma}\label{lem:irrhom}
Let $\lambda\in(0,1/2)$ be an irrational number and let
$J\subset(0,1/2)$ be  a closed interval. Suppose that
$z=(x,y)=g^{-1}(0,\lambda)$ for some $g\in SL_+(2,\Z)$ so that
$y>0$. Then there exist  natural numbers $a,b,c,d$ such that if
$z'=(x',y')=\big(h_{a,b,c,d}\big)^{-1}(x,y)$ with
\[h_{a,b,c,d}=(h^+)^{a}\cdot h^-\cdot h^+\cdot(h^-)^{b}\cdot h^+\cdot
h^-\cdot(h^+)^{c}\cdot(h^-)^{d}\]
then
\[y'\in J\quad\text{ and }\quad\big(h_{a,b,c,d}\big)_*(z')\,\beta=\beta.\]
Moreover, the number $a$ can be chosen arbitrary large.
\end{lemma}

\begin{remark}\label{rem:irratcoeff}
Note that the irrationality of $\lambda$ implies the irrationality
of $y$. If additionally the matrix $g$ has positive entries then
$x$ is also irrational.
\end{remark}

\begin{proof}[Proof of Lemma~\ref{lem:irrhom}]
Since $y$ is irrational (see  the previous Remark), each orbit of the rotation $\T\ni t \mapsto t-y\in\T$ is dense in $\T=[-1/2,1/2)$, so there exists a
natural number $a$ such that $x-ay=-1/2+\vep_1 \mod 1$ in $\T$, where
\[0<\vep_1<\frac{1}{2}\min\big(y,1/2-y\big).\]
In addition, by Lemma~\ref{lem:mn}, one can assume that the number $a$ is
arbitrarily large and $a': =m_a^+(z)$ is positive. Then
\[z_1=(x_1,y_1):=(h^+)^{-a}(x,y)=(x-ay\modulo 1,y\modulo 1)=(-1/2+\vep_1,y)\]
and
\begin{equation}\label{eq:homz1}
(h^+)^{a}_*(z_1)=(h^+)^{a'}.
\end{equation}
Let
\begin{align*}
z_2&=(x_2,y_2):=(h^-)^{-1}(x_1,y_1)=(x_1\modulo 1,y_1-x_1\modulo
1)\\&=(-1/2+\vep_1,y-1/2-\vep_1).
\end{align*}
Since $y<1/2$, we have $x_2+y_2=y-1<-1/2$, so that $z_2\notin S$.
Let
\begin{align*}
z_3&=(x_3,y_3):=(h^+)^{-1}(x_2,y_2)=(x_2-y_2\modulo 1,y_2\modulo
1)\\&=(2\vep_1-y,y-1/2-\vep_1).
\end{align*}
Since $0<\vep<1/2$, we have $x_3+y_3=\vep_1-1/2\in(-1/2,0)$, so that $z_3\in S$. Moreover, as $2\vep_1<y$, we have $x_3=2\vep_1-y<0$. Since $x_3$ is
irrational (see Remark~\ref{rem:irratcoeff} for the matrix $g\cdot (h^+)^a\cdot h^-\cdot h^+$ with has positive entries), by Lemma~\ref{lem:mn}, there
exists a natural  number $b$ such that $b':=m_b^-(z_3)>a'$ and $y_3-bx_3=1/2-\vep_2 \mod 1$ in $\T$, where
\[0<\vep_2<\frac{1}{2}\min\big(|x_3|,1/2-|x_3|\big)=\frac{1}{2}\min\big(y-2\vep_1,1/2-y+2\vep_1\big).\]
Therefore
\[z_4=(x_4,y_4):=(h^-)^{-b}(x_3,y_3)=(x_3\modulo 1,y_3-bx_3\modulo 1)=(2\vep_1-y,1/2-\vep_2)\]
and
\begin{equation}\label{eq:homz4}
(h^-)^{b}_*(z_4)=(h^-)^{b'}.
\end{equation}
Let
\begin{align*}
z_5&=(x_5,y_5):=(h^+)^{-1}(x_4,y_4)=(x_4-y_4\modulo 1,y_4\modulo 1)\\
&=(2\vep_1+1/2+\vep_2-y,1/2-\vep_2).
\end{align*}
Since $y<1/2$ and $\vep_1>0$, we have $x_5+y_5=2\vep_1+1-y>1/2$,
so that $z_5\notin S$. Let
\begin{align*}
z_6&=(x_6,y_6):=(h^-)^{-1}(x_5,y_5)=(x_5\modulo 1,y_5-x_5\modulo 1)\\
&=(2\vep_1+1/2+\vep_2-y,-2(\vep_2+\vep_1)+y).
\end{align*}
Since $0<\vep_2<1/2$, we have $x_6+y_6=1/2-\vep_2\in(0,1/2)$, so
that $z_6\in S$. Since $y_6$ is irrational (see
Remark~\ref{rem:irratcoeff}), by Lemma~\ref{lem:mn}, there exists
a natural number $c$ such that $c':=m_c^-(z_6)=b'-a'$. Setting
\[z_7=(x_7,y_7):=(h^+)^{-c}(x_6,y_6)\]
we have
\begin{equation}\label{eq:homz7}
(h^+)^{c}_*(z_7)=(h^+)^{c'}.
\end{equation}
Since $x_7$ is irrational (see Remark~\ref{rem:irratcoeff}), there
exists a natural number $d$ such that $y_7-dx_7\mod 1\in J$. Let
$d':=m^-_d(z_7)$. Setting
\[z'=(x',y'):=(h^-)^{-d}(x_7,y_7)=(x_7,y_7-dx_7\modulo 1),\]
we obtain $y'\in J$ and
\begin{equation}\label{eq:homzprim}
(h^-)^{d}_*(z')=(h^-)^{d'}.
\end{equation}
Next note that
\begin{align*}
\big(&h_{a,b,c,d}\big)_*(z')\\
&=(h^+)^{a}_*(z_1)\cdot h^-_*(z_2)\cdot h^+_*(z_3)
\cdot(h^-)^{b}_*(z_4)\cdot h^+_*(z_5)\cdot h^-_*(z_6)
\cdot(h^+)_*^{c}(z_7)\cdot(h^-)_*^{d}(z').
\end{align*}
Since $z_3,z_6\in S$ and $z_2,z_5\notin S$, by Lemma~\ref{lem:basic},
\[(h^+)_*(z_3)=h^+,\quad(h^-)_*(z_6)=h^-,\quad(h^-)_*(z_2)=(h^-)^{-1},\quad(h^+)_*(z_5)=(h^+)^{-1}.\]
If we combine this with \eqref{eq:homz1}-\eqref{eq:homzprim}, we obtain
\[\big(h_{a,b,c,d}\big)_*(z')=(h^+)^{a'}\cdot (h^-)^{-1}\cdot h^+
\cdot(h^-)^{b'}\cdot (h^+)^{-1}\cdot h^-
\cdot(h^+)^{c'}\cdot(h^-)^{d'}.\]
In view of \eqref{eq:matrixrel}, it follows that
\begin{align*}
\big(h_{a,b,c,d}\big)_*(z')&=(h^+)^{a'-1}\cdot \omega
\cdot(h^-)^{b'-1}\cdot \omega^{-1}
\cdot(h^+)^{c'}\cdot(h^-)^{d'}\\
&=(h^+)^{a'-1}\cdot(h^+)^{-b'+1}\cdot(h^+)^{c'}\cdot(h^-)^{d'}=(h^-)^{d'}.
\end{align*}
Consequently, $\big(h_{a,b,c,d}\big)_*(z')\,\beta=\beta$.
\end{proof}

\begin{proof}[Proof of Theorem~\ref{theorem:irrational}]
Let $J=[1/6,1/3]$. Using Lemma~\ref{lem:irrhom}, we can construct
a sequence $(z_n)_{n\geq 0}$ in $\T^2_0$ such that
$z_0=(0,\lambda)$ and for every $n\geq 1$ there exist natural
numbers $a_n,b_n,c_n,d_n$ with $a_n\geq 6$ such that
\begin{equation}\label{eq:abcdmain}
z_{n}=(h_{a_n,b_n,c_n,d_n})^{-1}z_{n-1}\quad\text{and}\quad(h_{a_n,b_n,c_n,d_n})_*(z_{n})\,\beta=\beta
\end{equation}
and $z_n=(x_n,y_n)\in(-1/2,1/2)\times(1/6,1/3).$

Let us consider the  irrational number
\[\alpha=[0;a_1,1,1,b_1,1,1,c_1,d_1,a_2,1,1,b_2,1,1,c_2,d_2,\ldots].\]
We will show that the directional flow along the vector $(1,\alpha)$ on
$\widetilde{M(0,\lambda)}_\beta$ is ergodic. In view of the proof of Lemma~\ref{lem:irrhom},
we have a freedom of choice of $d_n$ for fixed $a_n,b_n,c_n$. It follows that
the set of ergodic directions is uncountable.

Let $\alpha=[0;e_1,e_2,\ldots]$. In view of \eqref{eq:abcdmain}, we have
\[\big((h^{+})^{e_{1}}\cdots(h^{-})^{e_{8n}}\big)^{-1}\!
M(0,\lambda)=M(z_n),\
\big((h^{+})^{e_{1}}\cdots(h^{-})^{e_{8n}}\big)^{-1}_*(0,\lambda)\,
\beta=\beta.\]
Moreover,
\[y_n\in[1/6,1/3]\quad\text{and}\quad e_{8n+1}=a_{n+1}\geq 6\geq \frac{2}{1-2y_n}.\]
Finally Theorem~\ref{theorem:ergodicity} applied to the sequence $(k_n)_{n\geq 1}$, $k_n=8n$
yields the ergodicity of the directional flow along the vector $(1,\alpha)$.
\end{proof}

\section{Rational $\lambda$}\label{section:rational}
The aim of this section is to describe more precisely a subset
of ergodic directions on $\widetilde{M(0,\lambda)}_\beta$
whenever $\lambda\in(0,1/2)$ is rational. Such precise description
will help us to show that the Hausdorff dimension of the set of
ergodic directions is greater than $1/2$.

\begin{notation}
For every $\lambda=p/2q$ with $0<p<q$ relatively prime natural
numbers set
\begin{equation}\label{B}
B(\lambda):=\left\{
\begin{matrix}(3q-1,1,1,4q-1,1,1,q)& \text{ if $p$ is odd} \\
(2q+a,p-1,p+1,2q+2a, p-1,p+1,a) & \text{ if $p$ is even,}
\end{matrix}
\right.\end{equation}
where $a$ is the unique natural number satisfying $0<a\leq q$ and $ap=-1\mod q$.
\end{notation}

\begin{theorem}\label{theorem:rational}
Suppose that $\lambda=p/2q\in(0,1/2)$ with $p,q\in\Z$ relatively prime. For
every sequence of natural numbers $(n_k)_{k=1}^\infty$ if
\[\alpha=[0;B(\lambda),n_1,B(\lambda),n_2,\ldots,B(\lambda),n_k,\ldots]\]
then the directional flow along the vector $(1,\alpha)$ on
the $\mathbb{Z}$-cover $ \widetilde{M(0,\lambda)}_\beta$ given by $\beta$ is ergodic. Moreover, the
Hausdorff dimension of the set of such ergodic directions is
greater than $1/2$.
\end{theorem}

The proof of the first part of Theorem~\ref{theorem:rational} is rather technical
and is postponed until the end of the section, where its more general version
Themorem~\ref{theorem:rationalgen} is proved. The proof is based on two technical
Lemmas~\ref{lemma:oddgen} and \ref{lemma:evengen} in which we find explicitly some elements of the
Veech group of the translation surface $\widetilde{M(z)}_\beta$ when both coordinates of $z$
are rational.

The estimate from below of the Hausdorff dimension in Theorem~\ref{theorem:rational}
follows immediately form the following result, whose proof is fairly standard,
but is included for the convenience of the reader.
\begin{proposition}\label{proposition:dimension}
For any $\overline{a}\in\N^m$ ($m\geq 0$) a set $B=b\N+c\subset\N$ ($b,c$ are integer with $b>0$ and $c\geq 0$)
the Hausdorff dimension of the set
\[\mathcal{E}(\overline{a})=\Big\{[0;\overline{a},n_1,\overline{a},n_2,\overline{a},n_3,\ldots]:
n_i\in B\text{ for }i\geq 1\Big\}\]
is greater than $1/2$.
\end{proposition}
\begin{proof}
Let $\overline{a}=a_1\ldots a_{m}$. It simplifies the argument,
and causes no loss of generality, to assume that $m\geq 3$ is odd.
For every $l\in \N$ let us consider the map
$\psi_{\overline{a},l}:[0,1]\to[0,1]$ given by
\[\psi_{\overline{a},l}(x)=[0;a_1,a_2,\ldots, a_{m}, l+x]=\frac{p_m(\overline{a})(l+x)+p_{m-1}(\overline{a})}{q_m(\overline{a})(l+x)+q_{m-1}(\overline{a})}.\]
Then
\begin{equation}\label{eq:rangpsi}
\psi_{\overline{a},l}([0,1])=\big[[0;a_1,a_2,\ldots, a_{m}, l],[0;a_1,a_2,\ldots, a_{m},l+1]\big]
\end{equation}
and for every $x\in[0,1]$ we have
\[\psi'_{\overline{a},l}(x)=\frac{1}{(q_m(\overline{a})(l+x)+q_{m-1}(\overline{a}))^2}\geq \frac{1}{(q_m(\overline{a})(l+1)+q_{m-1}(\overline{a}))^2}=:d_{\overline{a},l}\]
and
\[\psi'_{\overline{a},l}(x)\leq \frac{1}{(q_m(\overline{a})l+q_{m-1}(\overline{a}))^2}<1/4.\]
For every $u\in\N$ let $B_u=b\{1,\ldots,u\}+c$. Then
\[\mathcal{E}(\overline{a})\supset \mathcal{E}_u(\overline{a})=\bigcap_{k\geq 1}\bigcup_{(n_1,\ldots,n_k)\in (B_u)^k}
\psi_{\overline{a},n_{1}}\circ\psi_{\overline{a},n_{2}}\circ\ldots\circ\psi_{\overline{a},n_{k}}[0,1].\]
Let
\[E_u:=\big[[0;a_1,a_2,\ldots, a_{m},1],[0;a_1,a_2,\ldots, a_{m},u+1]\big].\]
In view of \eqref{eq:rangpsi}, $\psi_{\overline{a},l}(E_u)\subset E_u$ for every $l\in B_u$ and the intervals $\psi_{\overline{a},l}(E_u)$, $l\in B_u$ are pairwise disjoint. In view of Theorem~9.7 in \cite{Falconer}, if $s_u>0$ is the unique solution of the equation
$\sum_{l=1}^ud_{\overline{a},bl+c}^s=1$,
then $\dim_H(\mathcal{E}_u(\overline{a}))\geq s_u$.
Since $\sum_{l=1}^{\infty}d_{\overline{a},bl+c}^{1/2}=+\infty$, we can choose $u\in\N$ such that $\sum_{l=1}^{u}d_{\overline{a},bl+c}^{1/2}>1$.
Thus $s_u>1/2$, and we get
\[\dim_H(\mathcal{E}(\overline{a}))\geq \dim_H(\mathcal{E}_u(\overline{a}))\geq s_u>1/2,\]
as desired.
\end{proof}

From now on, we will deal with square tiled translation surfaces $M(z)$ for which
$z=(r/2q,s/2q)\in \T_0^2$ with $r,s,q\in \Z$, $|r|,|s|<q$, $s$ is non-zero and coprime with $q$.

\begin{lemma}\label{lemma:oddgen}
Suppose that at least one number $s$ or $r$ is odd.
Let $a,b$ be natural numbers such that
\begin{equation}\label{eq:defab0}
0<a,b\leq 2q\quad\text{and}\quad r+as=-q\modulo 2q,\quad bs+s-q=r\modulo 2q.
\end{equation}
Then setting
\[g_z:=(h^+)^{2q+b}\cdot h^-\cdot h^+\cdot(h^-)^{2q+a+b}\cdot
h^+\cdot h^-\cdot(h^+)^{a}\in SL_+(2,\Z)\] we have
\[g_z\cdot M(z)=M(z)\ \text{ and }\ (g_z)_*(z)=id.\]
\end{lemma}

\begin{proof}
Without loss of generality we can assume that $s> 0$. If $s<0$ then dealing with
the point $-z=(-r/2q,-s/2q)\in \T_0^2$ we have
$g_{-z}\cdot M(-z)=M(-z)$ and $(g_{-z})_*(-z)=id$.
As $z\in E$, in view of
\eqref{eq:minusidtrans}, we obtain
$g_{-z}\cdot M(z)=M(z)$ and $(g_{-z})_*(z)=id$, which is our claim.

Next note that, by \eqref{eq:matrixrel},
\[g_z=(\vartheta\cdot\widetilde{g}_z\cdot\vartheta^{-1})\cdot\widetilde{g}_z, \quad \text{where}
\quad
\widetilde{g}_z=(h^-)^{2q+b}\cdot
h^+\cdot h^-\cdot(h^+)^{a}.\]
This symmetry of $g_z$ combined with the fundamental Lemma~\ref{lem:basic} is the heart of the proof.
Indeed, using Lemma~\ref{lem:basic} we will prove that
\begin{equation}\label{eq:przedsym}
\widetilde{g}_z\cdot M(z)=M(\vartheta z),\quad (\widetilde{g}_z)_*(z)=(h^-)^{\tilde{b}}\cdot \omega^{-1}\cdot(h^+)^{\tilde{a}}
\end{equation}
for some integer $\tilde{a}$, $\tilde{b}$. In view of \eqref{eq:przedsym}, \eqref{eq:vartheta} and \eqref{eq:matrixrel},
it follows that
\begin{align}\label{eq:przedsymtheta}
\begin{aligned}
(\vartheta\cdot\widetilde{g}_z\cdot\vartheta^{-1}) M(\vartheta z)&=M(z),\\ (\vartheta\cdot\widetilde{g}_z\cdot\vartheta^{-1})_*(\vartheta z)&=\vartheta\cdot(h^-)^{\tilde{b}}\cdot \omega^{-1}\cdot(h^+)^{\tilde{a}}\cdot\vartheta^{-1}=(h^+)^{\tilde{b}}\cdot \omega\cdot(h^-)^{\tilde{a}}.
\end{aligned}
\end{align}
Combining \eqref{eq:comphom} with \eqref{eq:przedsym} and \eqref{eq:przedsymtheta} and using again \eqref{eq:matrixrel}, we have
\[\big((\vartheta\cdot\widetilde{g}_z\cdot\vartheta^{-1})\cdot\widetilde{g}_z\big)\cdot M(z)=M(z),\]
\begin{align*}
\big((\vartheta&\cdot\widetilde{g}_z\cdot\vartheta^{-1})\cdot\widetilde{g}_z\big)_*(z)=
(h^+)^{\tilde{b}}\cdot \omega\cdot(h^-)^{\tilde{a}}\cdot(h^-)^{\tilde{b}}\cdot \omega^{-1}\cdot(h^+)^{\tilde{a}}\\
&=
(h^+)^{\tilde{b}}\cdot \omega\cdot(h^-)^{\tilde{a}+\tilde{b}}\cdot \omega^{-1}\cdot(h^+)^{\tilde{a}}
=
(h^+)^{\tilde{b}}\cdot(h^+)^{-(\tilde{a}+\tilde{b})}\cdot (h^+)^{\tilde{a}}=id,
\end{align*}
which is our claim.

It remains to prove \eqref{eq:przedsym}. Since $s$ and $q$ are
coprime and $s$ or $r$ are odd, there exist  $a,b,a',b'\in\Z$ with
\begin{equation}\label{eq:defab}
0<a,b\leq 2q\quad\text{and}\quad r+as=2qa'-q,\quad bs+s-q=2qb'+r.
\end{equation} Thus,
\begin{align*}
&(h^+)^{a}(r/2q,s/2q)=((r+as)/2q,s/2q)=(a'-1/2,s/2q),\\
&(h^-)^{2q}(s/2q,r/2q)=(s/2q,(r+2qs)/2q)=(s/2q,r/2q+s),\\
&(h^-)^{b}(s/2q,s/2q-1/2)=(s/2q,(bs+s-q)/2q)=(s/2q,r/2q+b').
\end{align*}
Hence, in view also of Lemma~\ref{lem:basic}, there exist
$k,l,m\in\Z$ such that
\begin{align}
\label{eq:f1gen}
&(h^+)^{a}M(z)=M(-1/2,s/2q),&&(h^+)^{q}_*(z)=(h^+)^{m},\\
\label{eq:f1+1/2gen}& (h^-)^{2q}M(\vartheta z)=M(\vartheta z),
&&(h^-)^{2q}_*(\vartheta z)=(h^-)^{k},\\
\label{eq:f2gen} &(h^-)^{b}M(s/2q,s/2q-1/2)=M(\vartheta z), &&
(h^-)^{b}_*(s/2q,s/2q-1/2)=(h^-)^{l}.
\end{align}
As $0< s<q$, we have $(-1/2,s/2q)\in S$ and $(-1/2,s/2q-1/2)\notin S$.
By  Lemma~\ref{lem:basic} and \eqref{eq:fund}, it follows that
\begin{align}
&h^-M\Big(-\frac{1}{2},\frac{s}{2q}\Big)=M\Big(-\frac{1}{2},\frac{s}{2q}-\frac{1}{2}\Big),\quad
(h^-)_*\Big(-\frac{1}{2},\frac{s}{2q}\Big)=h^-.\label{eq:f4gen}\\
&h^+M\Big(-\frac{1}{2},\frac{s}{2q}-\frac{1}{2}\Big)=M\Big(\frac{s}{2q},\frac{s}{2q}-\frac{1}{2}\Big),\quad
(h^+)_*\Big(-\frac{1}{2},\frac{s}{2q}-\frac{1}{2}\Big)=(h^+)^{-1}.\label{eq:f3gen}
\end{align}
Using
consequently \eqref{eq:f1gen}, \eqref{eq:f4gen}, \eqref{eq:f3gen},
\eqref{eq:f2gen}, \eqref{eq:f1+1/2gen} and
\eqref{eq:matrixrel}, we have
\[\widetilde{g}_z\cdot M(z)=((h^-)^{2q}\cdot(h^-)^{b}\cdot h^+\cdot h^-\cdot(h^+)^{a})\cdot M(z)=
M(\vartheta z)\] and
\begin{align*}(\widetilde{g}_z)_*(z)&= ((h^-)^{2q}\cdot(h^-)^{b}\cdot h^+\cdot
h^-\cdot(h^+)^{a})_*(z)\\&= (h^-)^{k}\cdot(h^-)^{l}\cdot
(h^+)^{-1}\cdot h^-\cdot (h^+)^{m}= (h^-)^{k+l-1}\cdot
\omega^{-1}\cdot (h^+)^{m}.
\end{align*}
This yields \eqref{eq:przedsym}, and the proof is complete.
\end{proof}
\begin{remark}\label{remark:odd}
Note that if $r=0$ then \eqref{eq:defab} holds for $a=q$ and $b=q-1$.
\end{remark}

\begin{lemma}\label{lemma:evengen}
Suppose that  $r$ and $s$ are both even. Let $a,b$ be natural numbers such that
\begin{equation}\label{congruences0}
0<a,b\leq q,\quad b=|s|,\quad\text{ and }\quad as+r=-1\mod q.
\end{equation}
Then setting
\[g_z:=(h^+)^{2q+a}\cdot(h^-)^{b-1}\cdot (h^+)^{b+1}\cdot(h^-)^{2q+2a}
\cdot(h^+)^{b-1}\cdot  (h^-)^{b+1}\cdot(h^+)^{a}\] we have
\[g_z \cdot M(z)=M(z)\quad\text{ and }\quad (g_z)_*(z)=id.\]
\end{lemma}

\begin{proof}
Without loss of generality we can assume
that $s>0$, otherwise we can use again \eqref{eq:minusidtrans} to pass to the positive case.

The general strategy is much the same as for the proof of Lemma~\ref{lemma:oddgen}.
Indeed, note that
\[g_z=(\vartheta\cdot\widetilde{g}_z\cdot\vartheta^{-1})\cdot\widetilde{g}_z, \quad \text{where}
\quad
\widetilde{g}_z=(h^-)^{2q+a}
\cdot(h^+)^{b-1}\cdot  (h^-)^{b+1}\cdot(h^+)^{a}\]
and we only need to show
\begin{equation}\label{eq:przedsymeven}
\widetilde{g}_z\cdot M(z)=M(-\vartheta z),\quad (\widetilde{g}_z)_*(z)=(h^-)^{\tilde{b}}\cdot \omega^{-1}\cdot(h^+)^{\tilde{a}}
\end{equation}
for some integer $\tilde{a}$, $\tilde{b}$.
Indeed, applying \eqref{eq:vartheta} together
with \eqref{eq:matrixrel} we have
\begin{equation*}
(\vartheta\cdot\widetilde{g}_z\cdot\vartheta^{-1}) M(\vartheta z)=M(-z),\quad (\vartheta\cdot\widetilde{g}_z\cdot\vartheta^{-1})_*(\vartheta z)=(h^+)^{\tilde{b}}\cdot \omega\cdot(h^-)^{\tilde{a}}.
\end{equation*}
Since $\vartheta z,-z\in E$, in view of \eqref{eq:minusidtrans}, it follows that
\begin{equation}\label{eq:przedsymthetaeven}
(\vartheta\cdot\widetilde{g}_z\cdot\vartheta^{-1}) M(-\vartheta z)=M(z),\quad (\vartheta\cdot\widetilde{g}_z\cdot\vartheta^{-1})_*(-\vartheta z)=(h^+)^{\tilde{b}}\cdot \omega\cdot(h^-)^{\tilde{a}}.
\end{equation}
Combining \eqref{eq:comphom} with \eqref{eq:przedsymeven} and \eqref{eq:przedsymthetaeven}, and using again \eqref{eq:matrixrel}, we have
\[\big((\vartheta\cdot\widetilde{g}_z\cdot\vartheta^{-1})\cdot\widetilde{g}_z\big)\cdot M(z)=M(z),\]
\begin{align*}
\big((\vartheta\cdot\widetilde{g}_z\cdot\vartheta^{-1})\cdot\widetilde{g}_z\big)_*(z)=
(h^+)^{\tilde{b}}\cdot \omega\cdot(h^-)^{\tilde{a}}\cdot(h^-)^{\tilde{b}}\cdot \omega^{-1}\cdot(h^+)^{\tilde{a}}=id,
\end{align*}
which is our claim.

It remains to prove \eqref{eq:przedsymeven}. As $s$ is even and
coprime with $q$, $q$ is odd. Since $r+1$ is also odd, there exist
integers $a,a'$ with
\begin{equation}\label{congruences}
a'\text{ is odd },\quad 0<a\leq q\quad\text{ and }\quad as+r=a'q-1,
\end{equation}
and set $b:=s$. Since
\begin{align*}
(h^+)^{a}(r/2q,s/2q)&=((r+as)/2q,s/2q)=(-1/2q+a'/2,s/2q),\\
(h^-)^{2q}(-s/2q,-r/2q)&=(-s/2q,-r/2q-s),\\
(h^-)^{a}(-s/2q,1/2-1/2q)&=(-s/2q,(q-1-as)/2q)\!=\!(-s/2q,(1-a')/2-r/2q),
\end{align*}
by Lemma~\ref{lem:basic} and \eqref{congruences}, there exists
$k,l,m\in\Z$ such that
\begin{align}
\label{eq:f10gen} &(h^+)^{a}M(z)=M(1/2-{1}/{2q},{s}/{2q}),\quad &&
(h^+)^{a}_*(z)=(h^+)^{m}
\\
\label{eq:f10+1/2gen}& (h^-)^{2q}M(-\vartheta z)=M(-\vartheta z),\quad &&
(h^-)^{2q}_*(-\vartheta z)=(h^-)^{k}
\\
\label{eq:f11gen}
& (h^-)^{a}M\Big(-\frac{s}{2q},\frac{1}{2}-\frac{1}{2q}\Big)=M(-\vartheta
z),\qquad &&
(h^-)^{a}_*\Big(-\frac{s}{2q},\frac{1}{2}-\frac{1}{2q}\Big)=(h^-)^{l}.
\end{align}
One can verify that, for $i,j \in \mathbb{Z}$ with $0\leq i \leq j < q$, one has
\begin{equation*}
\begin{split}
(h^+)^i \Big(\frac{j}{2q},\frac{1}{2}-\frac{1}{2q}\Big)&=\Big(\frac{j}{2q}+\frac{i}{2}-\frac{i}{2q}\modulo 1 \, ,\frac{1}{2}-\frac{1}{2q}\modulo 1\, \Big) \\
& = \left\{ \begin{array}{ll} \left(  \frac{j-i}{2q}, \frac{1}{2}-\frac{1}{2q}\right), & \textrm{for $i$ even, } \\ \left(  \frac{j-i}{2q} - \frac{1}{2}, \frac{1}{2}-\frac{1}{2q}\right), & \textrm{for $i$ odd. }  \end{array}\right.
\end{split}
\end{equation*}
If in addition $i<j$, one can also check that
\begin{equation*}
\left(  \frac{j-i}{2q}, \frac{1}{2}-\frac{1}{2q}\right) \notin S, \qquad   \left(  \frac{j-i}{2q} - \frac{1}{2}, \frac{1}{2}-\frac{1}{2q}\right) \in S.
\end{equation*}
Thus, by Lemma~\ref{lem:basic}, we have
\begin{equation*}
(h^+)_*   \left( (h^+)^i \left(\frac{j}{2q},\frac{1}{2}-\frac{1}{2q}\right)\right)
=\left\{ \begin{array}{ll} (h^+)^{-1} & \textrm{for $i$ even}, \\ h^+ & \textrm{for
 $i$ odd}  \end{array}\right.
\end{equation*}
Since  $1\leq b=s<q$ and $s$ is even by assumption, it follows by induction that
\begin{equation}
\begin{split}
\label{eq:f14gen}
&(h^+)^bM\Big(\frac{s}{2q},\frac{1}{2}-\frac{1}{2q}\Big)=M\Big(0,\frac{1}{2}-\frac{1}{2q}\Big),\\
&(h^+)^b_*\Big(\frac{s}{2q},\frac{1}{2}-\frac{1}{2q}\Big)=id;
\end{split}
\end{equation}
\begin{equation}
\begin{split}\label{eq:f15gen}
&(h^+)^{b-1}M\Big(\frac{1}{2}-\frac{1}{2q},\frac{1}{2}-\frac{1}{2q}\Big)=M\Big(-\frac{s}{2q},\frac{1}{2}-\frac{1}{2q}\Big),\\
&(h^+)^{b-1}_*\Big(\frac{1}{2}-\frac{1}{2q},\frac{1}{2}-\frac{1}{2q}\Big)=(h^+)^{-1}.
\end{split}
\end{equation}
Thus, recalling \eqref{eq:matrixrel} and applying \eqref{eq:vartheta} to
\eqref{eq:f14gen}, we have
\begin{align}
\begin{split}\label{eq:f16gen}
(h^-)^bM\Big(\frac{1}{2}-\frac{1}{2q},\frac{s}{2q}\Big)&=M\Big(\frac{1}{2}-\frac{1}{2q},0\Big),\quad
(h^-)^b_*\Big(\frac{1}{2}-\frac{1}{2q},\frac{s}{2q}\Big)=id.
\end{split}
\end{align}
In view of Lemma~\ref{lem:basic}, we have also
\begin{align}
\begin{split}\label{eq:f17gen}
h^-M\Big(\frac{1}{2}-\frac{1}{2q},0\Big)&=M\Big(\frac{1}{2}-\frac{1}{2q},\frac{1}{2}-\frac{1}{2q}\Big),\quad
(h^-)_*\Big(\frac{1}{2}-\frac{1}{2q},0\Big)=h^-.
\end{split}
\end{align}
Thus, using consequently \eqref{eq:f10gen}, \eqref{eq:f16gen},
\eqref{eq:f17gen}, \eqref{eq:f15gen}, \eqref{eq:f11gen}
and  \eqref{eq:f10+1/2gen}, 
 we have
\begin{equation*}\label{eq:f18gen}
(\widetilde{g}_z)\cdot M(z)=((h^-)^{2q+a}\cdot(h^+)^{b-1}\cdot h^-\cdot
(h^-)^b\cdot(h^+)^{a})\cdot M(z)= M(-\vartheta z)
\end{equation*} and using the same series of equations together with \eqref{eq:matrixrel} we also have
\begin{align*}
\begin{split}\label{eq:f19gen}
(\widetilde{g}_z)_*(z)&=((h^-)^{2q+a}\cdot(h^+)^{b-1}\cdot h^-\cdot
(h^-)^b\cdot(h^+)^{a})_*(z)\\&= (h^-)^{k+l}\cdot (h^+)^{-1}\cdot
h^-\cdot id\cdot (h^+)^{m}= (h^-)^{k+l-1}\cdot \omega^{-1}\cdot
(h^+)^{m}.
\end{split}
\end{align*}
This yields \eqref{eq:przedsymeven} and completes the proof.
\end{proof}

\begin{notation}
Let $z=(r/2q,s/2q)\in \T_0^2$, where $r,s,q$ are integer numbers with
$|r|,|s|<q$ and such that $s\neq 0$ is coprime with $q$.
Suppose that
\begin{equation}\label{B1}
B(z):=\left\{
\begin{matrix}(2q+b,1,1,2q+a+b,1,1,a)& \text{ if $s$ or $r$ is odd} \\
(2q+a,b-1,b+1,2q+2a, b-1,b+1,a) & \text{ if $s$ and $r$ are even,}
\end{matrix}
\right.\end{equation}
where $a,b$ are the natural numbers satisfying \eqref{eq:defab0} or
\eqref{congruences0} respectively.
\end{notation}

The following result is a more general version of Theorem~\ref{theorem:rational}.
\begin{theorem}\label{theorem:rationalgen}
Suppose that $z=(r/2q,s/2q)\in \T_0^2$, where $r,s,q$ are integer numbers with
$|r|,|s|<q$.  Additionally, assume that  $s\neq 0$ is coprime with $q$. For
every sequence of natural numbers $(n_k)_{k=1}^\infty$ in $2q\N$ if
\[\alpha=[0;B(z),n_1,B(z),n_2,\ldots,B(z),n_k,\ldots]\]
then the directional flow along the vector $(1,\alpha)$ on
the $\mathbb{Z}$-cover $ \widetilde{M(z)}_\beta$ given by $\beta$ is ergodic. Moreover, the
Hausdorff dimension of the set of such ergodic directions is
greater than $1/2$.

Moreover, if $r=0$ then ergodicity holds also for each sequence of natural numbers $(n_k)_{k=1}^\infty$.
\end{theorem}
\begin{proof}
By Lemma~\ref{lem:basic},
\begin{equation}\label{eq:hminus}
(h^-)^n\cdot M(z)=M(z)\quad\text{ and }\quad (h^-)^n_*(z)\ \beta=\beta
\end{equation}
for every $n\in 2q\N$. Moreover, if $r=0$ then \eqref{eq:hminus} is valid for every natural $n$.
In view of Lemmas~\ref{lemma:oddgen} and \ref{lemma:evengen}, it follows that
\[( g_z\cdot(h^-)^n)\cdot M(z)=M(z)\text{ and } ( g_z\cdot (h^-)^n)_*(z)\, \beta=\beta.\]
Setting $\alpha=[0,a_1,a_2,a_3,\ldots]$, it follows that
\begin{align*}
((h^{+})^{a_{1}}\cdots(h^{-})^{a_{8k}})^{-1}\cdot
M(z)&=M(z)\\
((h^{+})^{a_{1}}\cdots(h^{-})^{a_{8k}})^{-1}_*(z)\
\beta&=\beta.
\end{align*}
for every  natural $k$. Moreover,
\[a_{8k+1}>2q\geq\frac{2q}{q-s}=\frac{2}{1-2\frac{s}{2q}}.\]
In view of Theorem~\ref{theorem:ergodicity}, this gives the
ergodicity of the directional flow in the direction of $(1,\alpha)$
on $\widetilde{M(z)}_\beta$. The lower bound on the
Hausdorff dimension of the set of ergodic directions then follows
directly from Proposition~\ref{proposition:dimension}.
\end{proof}

\section{Acknowledgements}
We would like to thank Pascal Hubert and Barak Weiss who  gave us
preliminary versions of the paper \cite{Hu-We} and explained us
Theorem  \ref{thm:Hu-We} and Vincent Delecroix for useful
discussions. This collaboration was partially supported by the
EPSRC Grant EP/I019030/1 and by the Narodowe Centrum Nauki Grant
DEC-2011/03/B/ST1/00407.

\end{document}